\documentclass{article}
\usepackage{graphicx,amsmath,amssymb,amsfonts,amsthm,
comment,color,bm} 

\usepackage[algoruled]{algorithm2e}
\usepackage{tikz}
\usepackage{pgfplots}
\usepackage{pgfplotstable}

\newtheorem{theorem}{Theorem}

\newtheorem{corollary}{Corollary}

\newtheorem{remark}{Remark}

\newcommand{\offdiag}{{\rm offdiag}}
\newcommand{\ones}{{\bf 1}}
\newcommand{\bv}{{\bm v}}
\newcommand{\bw}{{\bm w}}
\newcommand{\bx}{{\bm x}}

\newcommand{\bd}{{\bm d}}

\newcommand{\be}{{\bm e}}
\newcommand{\bz}{{\bm z}}

\newcommand{\bb}{{\bm b}}
\newcommand{\bsu}{{\bm u}}
\newcommand{\wh}{\widehat}

\definecolor{mypink1}{rgb}{0.858, 0.188, 0.478}
\definecolor{mypink2}{RGB}{219, 48, 122}
\definecolor{mypink3}{cmyk}{0, 0.7808, 0.4429, 0.1412}
\definecolor{mygray}{gray}{0.7}
\definecolor{mygray1}{gray}{0.6}
\definecolor{myred}{rgb}{1,0,1}

\DeclareMathOperator{\opvec}{vec}

\title{Component-wise accurate fixed point iterations for computing the square root of a singular M-matrix}

\author{Dario A. Bini\thanks{Dipartimento di Matematica Universit\`a di Pisa} \and Bruno Iannazzo\thanks{Università degli Studi di Perugia}\and Beatrice Meini\thanks{Dipartimento di Matematica Universit\`a di Pisa} \and Jie Meng\thanks{Ocean University of China, Tsingtao}}

\begin{document}
	\maketitle

\begin{abstract}
We analyze two fixed-point iterations for computing the principal square root of an M-matrix $A$. Although these iterations, with customary initialization, converge sublinearly when $A$ is a singular M-matrix, we show that, under suitable mild conditions on the initial approximation,  the convergence is linear. Moreover, we provide component-wise accurate versions of these iterations, which allow us to approximate the principal square root with a component-wise relative error uniformly bounded by a small multiple of the machine precision. Numerical experiments demonstrating the effectiveness of the proposed algorithms for certain classes of problems are presented.
\end{abstract}

\maketitle Binomial iteration

\section{Introduction}\label{sec:intro}
We are interested in the component-wise accurate computation of the principal square root $A^{1/2}$ of a singular M-matrix $A$. The matrix $A^{1/2}$ exists under the mild assumption that 0 is a semisimple eigenvalue of $A$; in this case, $A^{1/2}$ is a singular M-matrix as well.
The square root and inverse square root of discretized differential operators  that are M-matrices arise when computing the Dirichlet-to-Neumann and Neumann-to-Dirichlet maps \cite{druskin}.
Computing the principal matrix square root of a  singular M-matrix also finds applications in the analysis of fractional diffusion problems in complex network analysis \cite{fabio} as well as in scientific computing and engineering problems \cite{frommer}. 

An M-matrix is a matrix of the form $A=\alpha I -B$, where $B\ge 0$, $\alpha>0$ and $\rho(B)\le \alpha$. The M-matrix $A$ satisfies {\em property C} if either $\det A\ne 0$ or 0 is a semi-simple eigenvalue of $A$.  Nonsingular M-matrices and irreducible singular M-matrices satisfy property C.
Interesting examples in the latter class are Laplacian matrices of a connected graph $\mathcal G$, i.e., matrices $A=D-W$, where $W$ is the (irreducible) adjacency matrix of the graph and  $D=\operatorname{diag}(d_1,\ldots,d_n)$, $\bd=(d_i):=W\ones$, $\ones=(1,1,\ldots, 1)^\top$.

If $A$ is an M-matrix, then it can be written in the form $A=s(I-N)$, where $s>0$, $N\ge 0$, and $\rho(N)\le 1$; moreover, as pointed out in \cite{as82}, we may assume w.l.o.g.~that $\hbox{diag}(N)>0$, i.e., $N$ has positive diagonal entries. 
In \cite[Theorem 4]{as82}, it is shown that, if $A$ has property C, then  $A^{1/2}=s^{1/2}(I-G)$, where $G\ge 0$ solves the matrix equation $X^2-2X+N=0$ and $\rho(G)\le 1$, and is the limit of the sequence $X_{k+1}=\frac12(X_k^2+N)$, $X_0=0$.
This sequence, called {\em binomial iteration} in \cite[Section 6.8.1]{higham:book}, is formed by (component-wise) nonnegative matrices. The operations involved in computing $X_{k+1}$ from $X_k$ are multiplications and additions of nonnegative matrices; therefore, they are component-wise numerically stable.
The convergence of the binomial iteration is linear and turns sublinear for a generic initial approximation when $A$ is singular. Therefore, the binomial iteration  becomes of scarce practical use when $A$ is singular.

A different approach is proposed in \cite{uno}, where the authors have introduced iterative algorithms  based on the triplet representation of M-matrices \cite{alfa}, \cite{mehdi}  and on a suitable adaptation of Cyclic Reduction (CR), which is a quadratically convergent iteration widely used in the numerical solution of matrix equations arising in Markov chain models \cite[Chapter 7]{blm:book}.
If $A$ is singular, then the convergence of CR deteriorates from quadratic to linear with convergence rate $\frac{1}{2}$ as shown in \cite{guo}. As a remedy to restore the quadratic convergence, in \cite{uno} the shift technique of \cite[Section 3.6]{blm:book} is applied, preserving, even though partially, the component-wise accuracy of the algorithm. Unfortunately, this strategy can be applied only under strong assumptions on the matrix $A$.

Here, we analyze the natural iteration (that is, the aforementioned binomial iteration for $X_0$ not necessarily 0) and the U-based iteration \cite[Section 6]{blm:book} for computing the matrix $G$ such that  $A^{1/2}=s^{1/2}(I-G)$ when $A=s(I-N)$ is a singular M-matrix. Compared to CR, these two fixed-point iterations offer a significantly lower computational cost per step. However, their sublinear convergence for almost any choice of $X_0$ is a substantial limitation, rendering them impractical for most applications.

We show that with a suitable choice of $X_0$, the sequence $\{X_k\}$ generated by the Binomial and U-based iterations converges linearly, even when $A$ is singular, and the convergence rate can be explicitly expressed in terms of the eigenvalues of $A$.
Furthermore, we provide versions of these algorithms designed to avoid numerical cancellation at every step. This property leads to component-wise accurate algorithms for computing $A^{1/2}$ with respect to the relative error. Since the cost of each step is much smaller than the cost of a step of CR, these iterations may be a valid alternative to CR for computing $A^{1/2}$ in a stable and fast way in the case where $A$ is singular.

The effectiveness of the proposed fixed-point iterations is illustrated through numerical experiments on sparse or bipartite graph Laplacians, a setting where the accelerated CR version from \cite{uno} cannot be applied.

The paper is organized as follows. In Section~\ref{sec:prel}, we recall some preliminary results, in particular the GTH-like algorithm for the component-wise accurate solution of a linear system with an M-matrix, and the definition of triplet representation. In Section~\ref{sec:reduction}, we show that if $A=I-N$ is an M-matrix satisfying property C, then $A^{1/2}=I-G$ where $G$ is the {\em minimal nonnegative} solution of the equation $X^2-2X+N=0$. Section~\ref{sec:natural} is devoted to the analysis of the natural iteration in the case where $A$ is a singular M-matrix. Conditions on $X_0$ are given so that the convergence of the natural iteration is linear. The rate of convergence is explicitly given, in terms of the eigenvalues of $A$, in the case where $X_0\bsu=\bsu$ for $A\bsu=0$, and in the case where, in addition, $\bv^\top X_0=\bv^\top$ for $\bv^\top A=0$; a simple choice satisfying the above properties is $X_0=I$.
Section~\ref{sec:fixpU} contains a similar analysis for the U-based iteration. Section 6 deals with error propagation, along with comments on the overall numerical stability of fixed-point iterations. Section~\ref{11} gives necessary and sufficient conditions for the existence of a vector $\bsu>0$ such that $A\bsu=0$, where $A$ is a singular M-matrix,  and provides a general method for computing the matrix square root that works for singular M-matrices with a triplet.
Finally, Section~\ref{sec:exp} reports the results of some numerical tests, while Section~\ref{sec:conclusions} concludes the paper.

\section{Preliminaries}\label{sec:prel}
We know that for a matrix $A$ satisfying property C, there exists a unique square root, denoted by $A^{1/2}$, that is
{\em principal}, i.e., its eigenvalues have nonnegative real parts; moreover $A^{1/2}$ is {\em primary}, i.e., it is a polynomial of $A$, and is an M-matrix \cite{as82,uno}. This way, $A^{1/2}$ can be written as $A^{1/2}=s^{1/2}(I-G)$, $G\ge 0$.

Throughout the paper, we assume that $A$ is an M-matrix satisfying property C and, without loss of generality, we assume that $s=1$ so that we may write $A=I-N$, $N\ge 0$, $\hbox{diag}(N)>0$, and $A^{1/2}=I-G$ with $G\ge 0$, $\rho(G)\le 1$. 

\subsection{Component-wise accurate solution of a system with an M-matrix}

For a square matrix $M$, we write $\offdiag(M)$ for the matrix with zero diagonal entries and off-diagonal entries equal to those of $M$. We recall \cite{mehdi,li12} that the \emph{triplet representation} of an M-matrix $A$ is a triple $(-\offdiag(A), \bsu,\bv)$, where $\bsu$ and $\bv$ are nonnegative and positive vectors, respectively, such that $A \bsu = \bv$. From a triplet representation, it is possible to recover the diagonal entries of $A$ by avoiding numerical cancellation, namely $a_{ii}=(v_i -\sum_{j\ne i} a_{ij}u_j)/u_i$, $i=1,\ldots,n$.

Let $A$ be an invertible M-matrix, and assume we are given a triplet representation 
 $(-\offdiag(A), \bsu,\bv)$.
 Then the solution of the system $A\bx=\bb$, for $\bb\ge 0$,  can be computed in a component-wise accurate way by means of Algorithm~\ref{alg:gth}, based on the extension of the GTH algorithm \cite{alfa} and here given in the version described in \cite{poloni15}. 
\begin{algorithm}
   \textbf{Input}:
       $B\in\mathbb R^{n\times m}$, $B\ge 0$, and a triplet representation 
 $(-\offdiag(A), \bsu,\bv)$ of $A$ \;
   \textbf{Output}: $X = A^{-1} B$\;
   Set $L=I$,  $U=\offdiag(A)$\;
   \For{
      $\ell=1,2,\dots,n-1$}{
      $U_{\ell,\ell} = (v_\ell - U_{\ell,\ell+1:n} u_{\ell+1:n} )/u_\ell$\;
      $L_{\ell+1:n,\ell} = U_{\ell+1:n,\ell} /U_{\ell,\ell}$\;
      $v_{\ell+1:n} = v_{\ell+1:n} - L_{\ell+1:n,\ell} v_\ell$\;
      $U_{\ell+1:n,\ell+1:n} = \offdiag(U_{\ell+1:n,\ell+1:n} - L_{\ell+1:n,\ell} U_{\ell,\ell+1:n} )$\;
      $U_{\ell+1:n,\ell} = 0$\;
      }
   $U_{n,n}=v_n/u_n$\;
   Compute $Y = L^{-1} B$ by forward substitution\;
   Compute $X = U^{-1}Y$ by back substitution.
   \caption{GTH-like algorithm for solving the system $AX = B$, where $A\in\mathbb R^{n\times n}$ is a nonsingular M-matrix,  $X,B\in\mathbb R^{n\times m}$, and $B\ge 0$} \label{alg:gth}
\end{algorithm}

\section{Reduction to a quadratic matrix equation}\label{sec:reduction}
A formal manipulation of the equation $(I-X)^2=A$, where $A=I-N$, shows that the matrix $G=I-A^{1/2}$ solves the quadratic equation
\begin{equation}\label{eq:qme}
    X^2 - 2 X + N = 0.
\end{equation}

We show that $G$ is the entry-wise minimal nonnegative solution of \eqref{eq:qme}. 

\begin{theorem} 
\label{th:min}
    Let $A=I-N$, $N\ge 0$, $\mathrm{diag}(N)>0$ and $\rho(N)\le 1$, be an M-matrix with property C.
    Then, $A^{1/2}$  is an M-matrix of the form $I-G$, $G\ge 0$, $\rho(G)\le 1$, where
    $G=\lim_{k\to\infty}G_k$ with
    \[
    G_{k+1}=\frac{1}{2}(G_k^2+N),\quad G_0=0.
    \]
    Moreover, $G$ is the minimal nonnegative solution of \eqref{eq:qme}. Finally, if $A$ is singular then $\rho(G)=1$, and if $\bsu>0$ is such that $(I-N)\bsu=0$, then $G\bsu=\bsu$. If, in addition, $1$ is a simple eigenvalue of $N$, then $1$ is the only eigenvalue of $G$ of modulus $1$.
\end{theorem}
\begin{proof}
   Under the assumptions on $A$, according to \cite[Theorem 4]{as82}, $A$ has an M-matrix square root of the form $I-G$ with $G\ge 0$, $\rho(G)\le 1$ and $G=\lim_k G_k$. Moreover, $I-G$ is nonsingular or the eigenvalue $0$ is semisimple. Suppose that $0\le Y$ is a solution; we note that $0\le Q\le Y$ implies $0\le \frac{1}{2}(Q^2+N)\le \frac{1}{2}(Y^2+N)=Y$, and this inequality can be used as an inductive step to prove that $G\le Y$. In fact, observe that $0=G_0\le Y$, moreover, applying the inequality $\frac{1}{2}(Q^2+N)\le Y$ with $Q=G_k$, we find that $G_k\le Y$ implies $G_{k+1}\le Y$. By induction, this implication proves that $G_k\le Y$ for any $k\ge 0$ so that $G=\lim_k G_k\le Y$.
   This proves the minimality of $G$.
   
   According to \cite[Theorem 7]{uno}, $A^{1/2}$ is the unique M-matrix square root of $A$ that is a primary function of $A$ (namely a polynomial of $A$). On the other hand, $I-G$ is an M-matrix square root nonsingular or singular with semisimple $0$ eigenvalue and by \cite[Theorem 6.1]{u} it is a primary function of $A$. This implies $A^{1/2}=I-G$.
   
 We can write $G=I-A^{1/2}=I-(I-N)^{1/2}$ and if $A\bsu =0$, then $G\bsu =\bsu$, $1$ is an eigenvalue of $G$, and since $\rho(G)\le 1$, we get $\rho(G)=1$.
   Now, assume that $1$ is a simple eigenvalue of $N$. Denote by $\gamma_i$, $i=1,\ldots,n$, the eigenvalues of $N$, with the convention that $\gamma_1=1$. 
   The spectral mapping theorem \cite[Theorem 1.13(d)]{higham:book} guarantees that the eigenvalues of $G$ are $\lambda_i=1-(1-\gamma_i)^{1/2}$ and since
   $\gamma_i\ne 1$ and $|\gamma_i|\le 1$, for $i=2,\ldots,n$, then $|\lambda_i|<1$ for $i=2,\ldots,n$. Therefore, $1$ is the only eigenvalue of $G$ of modulus $1$.
\end{proof}

It is interesting to observe that part of the results given in Theorem \ref{th:min} is valid under more general assumptions, as expressed by the following.

\begin{remark}
   It is possible to prove that, given the more general splitting $A=M-N$, where $M$ is an M-matrix with property C, and $N\ge 0$, the square root of $A$ can be written as $A^{1/2}=M^{1/2}-W$, where $W$ is the minimal nonnegative solution of the algebraic Riccati equation $ X^2-M^{1/2}X-XM^{1/2}+N=0$.
\end{remark}

In the next sections, we analyze two fixed-point iterations for computing the matrix square root of an M-matrix satisfying property C in a numerically stable way with respect to the component-wise relative error.

\section{The Natural iteration}\label{sec:natural}
A first fixed point iteration, that resembles the {\em natural iteration} in the framework of Markov chains \cite[Chapter 6]{blm:book}, is given by
\begin{equation}\label{eq:nat}
   X_{k+1}=\frac{1}{2}(X_k^2+N),~k=0,1,\ldots,\quad X_0\ge0,~~X_0\bsu\le \bsu,
\end{equation}
where $\bsu>0$ is such that $(I-N)\bsu= 0$. If $X_0=0$, it coincides with the Binomial iteration \cite[Section 6.8.1]{higham:book}, implicitly introduced in Theorem \ref{th:min}.
The matrix sequence obtained with  $X_0=0$ 
converges linearly when $A$ is nonsingular, while the convergence is sublinear when $A$ is singular. 
In fact, the Jacobian of the fixed-point function $F(X)=\frac12(X^2+N)$ in the fixed point $G$ is given by $\frac12(I\otimes G+G^\top\otimes I)$; the spectral radius of the latter matrix is $\rho(G)$. Therefore, if $\rho(G)<1$, the convergence is linear with a converging rate $\rho(G)$, while if $\rho(G)=1$, then the convergence turns to sublinear.

However, we may show that if $X_0$ is such that  $X_0\bsu=\bsu$, say, $X_0=I$, the convergence turns to linear even though $A$ is singular.
We have the following convergence result

\begin{theorem}\label{th:convsingular}
   Let $A=I-N$, $N\ge 0$, be a singular M-matrix such that $N\bsu=\bsu$, for a given
   vector $\bsu>0$; moreover, assume that $1$ is a simple eigenvalue of $N$. Then,
   the sequence $\{X_k\}_k$ generated by \eqref{eq:nat} is convergent to the minimal nonnegative solution $G$ to the matrix equation \eqref{eq:qme}. Moreover, if $X_0\bsu = \bsu$,  
   then we have
   \[
       \limsup_{k\to\infty}\|X_k-G\|^{1/k}\le \sigma, \quad \sigma =\frac{1}{2}\max_{2\le i\le n}\max_{1\le j\le n}|\lambda_i+\lambda_j|<1,
   \]
   where $\lambda_i$, $i=1,\ldots,n$, are the eigenvalues of $G$ such that $1=\lambda_1>|\lambda_2|\ge \cdots \ge |\lambda_n|$ .
   Moreover, if in addition $\bv^\top X_0=\bv^\top$, where $\bv>0$ is such that $\bv^\top N=\bv^\top$,  then $\sigma=|\lambda_2|$.   
    \end{theorem}

\begin{proof}
    Under these assumptions, $A$ satisfies property C. We prove by induction that $X_k\ge 0$ and that $X_0\bsu=\bsu$ implies $X_k\bsu= \bsu$ for any $k\ge 0$, while $X_0\bsu\le \bsu$ implies $X_k\bsu\le\bsu$ for any $k\ge 0$. This is true for $k=0$; to prove the inductive step, assume that $X_k\ge 0$ and $X_k\bsu=\bsu$. We have $X_{k+1}\ge 0$ since it is the sum of nonnegative matrices; moreover, $X_{k+1}\bsu=\frac{1}{2}(X_k^2\bsu+N\bsu)= \bsu$ since $X_k^2\bsu=X_k(X_k\bsu)= X_k\bsu= \bsu$, and $N\bsu= \bsu$. Similarly, we do if $X_0\bsu\le\bsu$ to prove that $X_k\bsu\le\bsu$. From the conditions $X_k\ge 0$ and $X_k\bsu\le\bsu$, with $\bsu>0$, we deduce that the sequence $\{X_k\}_k$ is uniformly bounded from above so that it belongs to a compact set; therefore, it admits converging subsequences. Let us prove that the sequence $\{X_k\}_k$ has only one accumulation point, which coincides with the minimal nonnegative solution $G$ of \eqref{eq:qme}.
    Consider the sequence $\{G_k\}_k$ defined in Theorem \ref{th:min} satisfying $\lim _k G_k=G$;  since $X_0\ge G_0=0$, it follows by induction that $X_k\ge G_k$ for any $k\ge 0$.
    Let $\{X_{\sigma_k}\}_k$ be a subsequence of $\{X_k\}_k$ converging to some limit $H$. Then, from $0\le G_{\sigma_k}\le X_{\sigma_k}$ we get $0\le G\le H$. Since $X_{\sigma_k}\bsu\le \bsu$ then $0\le G\bsu \le H\bsu\le \bsu$. On the other hand, from Theorem~\ref{th:min}, we have $G\bsu=\bsu$, so that we get $H\bsu =G\bsu$. Since $H-G\ge 0$ and $\bsu>0$ we conclude that $H=G$.
    
    Concerning the convergence speed, observe that from the identity $G=\frac{1}{2}(G^2+ N)$ it follows that $G-X_{k+1}=\frac{1}{2}(G(G-X_k)+(G-X_k)X_k)$. That is, denoting
    $E_k:=G-X_k$, we have
    \[
       E_{k+1}=\frac{1}{2}(GE_k+E_kG-E_k^2).
    \]
    Setting $\be_k:=\hbox{vec}(E_k)$, where $\hbox{vec}(X)$ is the $n^2$-vector obtained by stacking the columns of $X$, we get
    \[
       \be_{k+1}=M_k\be_k,\quad M_k=\frac{1}{2}(I_n\otimes (G-E_k)+G^\top\otimes I_n).
    \]
    That is,
    \[
       \be_{k+1}=M_kM_{k-1}\cdots M_0\be_0.
    \]

    Now consider the Jordan form of $G$, that is $J=S^{-1} GS$. 
    Since $\lambda_1=1$ is a simple eigenvalue of $G$ by Theorem~\ref{th:min}, it appears in a Jordan block of size 1, so that we may assume $J$ in the form
    \begin{equation}\label{eq:jordan}
       J=\begin{bmatrix}
         1&0\\0&\widetilde J
       \end{bmatrix}
    \end{equation}
    and $\widetilde J$ has eigenvalues $\lambda_2,\ldots,\lambda_n$, moreover, $|\lambda_i|<1$ for $i>1$.

    Define $\widehat M_k:=(S^\top \otimes S^{-1}) M_k(S^\top\otimes S^{-1})^{-1}=\frac{1}{2}(I_n\otimes (J-S^{-1} E_k S)+J^\top\otimes I_n )$. 
    Since $\bsu=G\bsu=X_k\bsu$, we have $E_k\bsu=0$. Moreover, since the first column of $S $ is proportional to $\bsu$, then the first column of $\widehat E_k=S^{-1} E_k S$ is zero, so that the first $n$ components of     $\wh{\be}_k:=\hbox{vec}(\widehat E_k)$ are zero.
    Denoting $\widetilde{\be}_k$ the $(n^2-n)$ vector formed by the last $n^2-n$ component of $\widehat{\be}_k$, we get the recurrence
    \begin{equation}\label{eq:norme}
       \widetilde {\be}_{k+1}=\mathcal{M}_k\widetilde{\be}_0,\quad \mathcal {M}_k=\widetilde M_k \widetilde M_{k-1}\cdots \widetilde M_0,
    \end{equation}
    where $\widetilde M_k$ is the submatrix of $\widehat M_k$ formed by the last $n^2-n$ rows and columns, that is, $\widetilde M_k=\frac12(I_{n-1}\otimes (J- \widehat E_k)+\widetilde J^\top\otimes I_n)$.
    Rewrite the latter expression as 
    $\widetilde M_k=\widetilde M+\Delta_k$, 
    with $\widetilde M=\frac12(I_{n-1}\otimes J+\widetilde J^\top\otimes I_n)$, 
    $\Delta_k=-\frac12 ( I_{n-1}\otimes \widehat E_k)$.
    We show that, for any $\epsilon>0$, there exists an operator norm $\|\cdot\|_\epsilon$  and there exist $k_0$ such that
    \begin{equation}\label{eq:normeps}
       \| {\mathcal M}_k \|_\epsilon \le (\rho(\widetilde M) +\epsilon))^{k-k_0} \gamma_\epsilon,\quad k\ge k_0,
    \end{equation}
    where $\gamma_\epsilon$ is a positive constant independent of $k$ and $\rho(\cdot)$ denotes the spectral radius.
    To prove this bound, we apply the property that for any $\epsilon$ and for any matrix $W$, there exists an operator norm $\|\cdot\|_\epsilon$ such that $\|W\|_\epsilon\le\rho(W)+\epsilon/2$. Applying this property to $W=\widetilde M$, from \eqref{eq:norme} we get
    \[
       \|\mathcal M_k\|_\epsilon\le
       (\rho(\widetilde M)+\epsilon/2+\|\Delta_k\|_\epsilon)
       \cdots (\rho(\widetilde M)+\epsilon/2+\|\Delta_0\|_\epsilon).
    \]
    Now, since $\|\Delta_k\|_\epsilon$ tends to 0 for $k\to\infty$, there exists $k_0$ such that for any $k\ge k_0$ we have $\|\Delta_k\|_\epsilon\le \epsilon/2$. This fact implies that
    \[
       \|\mathcal M_k\|_\epsilon \le (\rho(\widetilde M)+\epsilon)^{k- k_0}\gamma_\epsilon , \quad \gamma_\epsilon =\prod_{i=0}^{k_0}\|\widetilde M_i\|_\epsilon. 
    \]
    For any other matrix norm $\|\cdot\|$, there exists a constant $\delta_\epsilon$ such that $\|W\|\le\delta_\epsilon \|W\|_\epsilon$ for any matrix $W$. Therefore, we obtain
    \[
        \|\mathcal M_k\|\le\gamma_\epsilon\delta_\epsilon (\rho(\widetilde M)+\epsilon)^{k-k_0}.
    \]
    so that
    \[
        \|\widetilde{\be}_{k+1}\|\le \gamma_\epsilon\delta_\epsilon\|\be_0\|(\rho(\widetilde M)+\epsilon)^{k-k_0}.
    \]
    Taking the $k+1$st root on both sides of the above expression, we get
    \[
        \limsup_{k\to\infty}\|\widetilde \be_{k+1}\|^{1/(k+1)}\le\rho(\widetilde M)+\epsilon,
        \hbox{ for any }\epsilon>0,
    \]
    that implies     $\limsup_{k\to\infty}\|\be_{k}\|^{1/k}\le\rho(\widetilde M)$.
    The thesis follows since the spectral radius  of  $ \widetilde M=\frac12(I_{n-1}\otimes J+\widetilde J\otimes I_n)$ is 
    $\sigma=\frac{1}{2} \max_{2\le i\le n}\max_{1\le j\le n}|\lambda_i+\lambda_j|$. Since $\lambda_1=1$ and for $i>1$, we have $\lambda_i\ne 1$, $|\lambda_i|\le 1$, then $\sigma<1$.

    Finally, if $X_0\bsu=\bsu$ and $\bv^\top X_0=\bv^\top$ then   the first row and the first column of $\widehat E_0$ have null entries. 
    Therefore $\widehat E_k$ has the first row and the first column equal to 0 for any $k\ge 0$. This implies that the first block formed by $n$ entries of $\widehat{\be}_k$ is zero, and the remaining blocks have their first component equal to zero.
    This allows us to repeat the same analysis as in the previous case by replacing the matrix $\widetilde M$ with $\widetilde M=\frac12(I_{n-1}\otimes\widetilde J+\widetilde J^\top\otimes I_{n-1})$. 
    The eigenvalues of this matrix are $\frac{1}{2}(\lambda_i+\lambda_j)$, for $i,j=2,\ldots, n$, so that $\sigma=|\lambda_2|$. This concludes the proof. 
\end{proof}

\begin{remark}\rm
Observe that in Theorem \ref{th:convsingular} we may replace the assumption that $1$ is a simple eigenvalue of $N$ with the condition that $1$ is semisimple. But in this case, the condition $X_0\bsu=\bsu$ must be satisfied for all the eigenvectors $\bsu$ of $N$ corresponding to the eigenvalue 1. Similarly, for the condition $\bv^\top X_0=\bv^\top$ of $N$ corresponding to the eigenvalue 1.
Observe also that the choice $X_0=I$ satisfies both conditions $X_0\bsu=\bsu$ and $\bv^\top X_0=\bv^\top$ for any right and left eigenvectors $\bsu$ and $\bv$ of $N$ corresponding to the eigenvalue $\lambda=1$.
\end{remark}

As pointed out in \cite{higham:book} for the case $X_0=0$, the floating-point implementation of \eqref{eq:nat} to compute an approximation to the matrix $G$ incurs no cancellation and is component-wise numerically stable. 
The potential cancellation in computing the diagonal entries of $Y = I-G$ can be easily overcome by using the triplet representation of $Y$. In fact, since $A^{1/2}\bsu=0$ and $G\bsu=\bsu$, for the diagonal entries $y_{11},\ldots, y_{nn}$ of $Y=A^{1/2}$ we have 
$y_{ii}u_i=-\sum_{j=1,\, j\ne i}^n y_{ij}u_j=\sum_{j=1,\, j\ne i}^n g_{ij}u_j$, where $G=(g_{ij})$. 
So that we obtain the expression
$y_{ii}=\frac{1}{u_i}\sum_{j=1,\, j\ne i}^n g_{ij}u_j$ whose implementation in floating point arithmetic does not involve cancellation.

The resulting procedure is reported as Algorithm \ref{alg:bin}.

\begin{algorithm} 
    \textbf{Input:} A matrix $N\ge 0$ with $\rho(N)=1$ and simple eigenvalue equal to $1$; a vector $\bsu>0$ such that $N\bsu=\bsu$; a matrix $X_0\ge 0$ such that $X_0\bsu=\bsu$; a tolerance $\epsilon>0$\;
    \textbf{Output:} An approximation $Y$ to $(I-N)^{1/2}$ together with the triplet representation $(-\offdiag(Y),\bsu,0)$ of $Y$ \;      
    \For {$k=0,1,2,\ldots$}
          {
          $X_{k+1}=\frac{1}{2}(X_k^2+N)$\;
          If {$|X_{k+1}-X_k|\le\epsilon|X_{k+1}|$}{\quad exit\;}
          }
    $-\offdiag(Y)=\offdiag(X_{k+1})$, $\hbox{diag}(Y)=-\hbox{diag}(\bsu)^{-1}\offdiag (Y)\bsu$.
    \caption{Natural (Binomial) fixed point iteration}\label{alg:bin}
\end{algorithm}

\section{The U-based iteration}\label{sec:fixpU}
A different iteration for computing the minimal nonnegative solution to \eqref{eq:qme} is the one derived from the  {\em U-based iteration} described in \cite[Chapter 6]{blm:book} and is given, for $k=0,1,\ldots$, by
\begin{equation}\label{eq:ub}
X_{k+1}=(2I-X_k)^{-1}N,\quad X_0\ge 0,\quad X_0\bsu \le \bsu.
\end{equation}
Its convergence is faster than the convergence of \eqref{eq:nat}, and its rate of convergence can be explicitly expressed  as shown in the following

\begin{theorem}\label{th:convu}
   Under the same assumptions of Theorem \ref{th:convsingular},
   the sequence $\{X_k\}_k$ generated by \eqref{eq:ub}
   is such that $X_k\ge 0$, $X_k\bsu\le \bsu$, and   $2I-X_k$ is a nonsingular M-matrix.
   Moreover, $\lim_{k\to\infty} X_k = G$, with
   $G$ the minimal nonnegative solution to the matrix equation \eqref{eq:qme}. If $X_0\bsu = \bsu$, then we have  $X_k\bsu= \bsu$ for any $k\ge 0$ and
   \[
       \limsup_{k\to\infty}\|X_k-G\|^{1/k}\le \sigma,
       \quad \sigma= |\lambda_2|,
   \]
   where $1=\lambda_1>|\lambda_2|\ge \cdots \ge |\lambda_n|$ are the eigenvalues of $G$. 
   Moreover, if in addition, $\bv^\top X_0=\bv^\top$, where $\bv>0$ is such that $\bv^\top N=\bv^\top$, then       
   \[
      \sigma= \max_{2\le j\le n}\left|\frac{\lambda_2}{2-\lambda_j}\right|.
   \]
\end{theorem}
\begin{proof}
   We proceed by induction on $k$ to show that $X_k\ge 0$, $X_k\bsu\le \bsu$, and   $2I-X_k$ is a nonsingular M-matrix.
   For $k=0$ we have $X_0\ge 0$, $X_0\bsu\le \bsu$, and $(2I-X_0)\bsu> \bsu$. Therefore, in view of \cite[Theorem 2.3, I27]{bp:book}, $2I-X_0$ is a nonsingular M-matrix. Now assume that $X_k\ge 0$, $X_k\bsu\le \bsu$,  and $2I-X_k$ is a nonsingular M-matrix.  Then, in view of \cite[Theorem 2.3, N38]{bp:book}, $(2I-X_k)^{-1}\ge 0 $ so that $X_{k+1}\ge 0$. Moreover,
   \[
      X_{k+1}\bsu=(2I-X_{k})^{-1}N\bsu=(2I-X_{k})^{-1}\bsu<\bsu,
   \]
   where the last equality follows from $(2I-X_k)\bsu>\bsu$. Then, it follows  $(2I-X_{k+1})\bsu>\bsu$. 
   Thus, in view of \cite[Theorem 2.3, I27]{bp:book}, $2I-X_{k+1}$ is a nonsingular M-matrix. If $X_0\bsu =\bsu$ we can prove, by following the same lines as in the proof of Theorem~\ref{th:convsingular}, that $X_k\bsu =\bsu$ for any $k\ge 1$. Similarly, we show the convergence to $G$.
    
   Concerning the convergence speed, performing a similar analysis as in the proof of Theorem~\ref{th:convsingular}, we arrive at 
   \[
      E_{k+1}=(2I-G+E_k)^{-1}E_kG, 
   \]
   for $E_k=G-X_k$, which yields
   \[
      \be_{k+1}=(G^\top\otimes  (2I-G+E_k)^{-1})\be_k.
   \]
   Performing a similar analysis as in the proof of Theorem \ref{th:convsingular}, where we denote $J$ the Jordan form of $G$ and $\widetilde J$ the trailing $(n-1)\times(n-1)$ submatrix of $J$ given in equation \eqref{eq:jordan}, we arrive at 
   \[
       \limsup_{k\to\infty}\|X_k-G\|^{1/k}=\rho(\widetilde{J} ^\top\otimes(2I-J))=\max_{2\le i\le n}\max_{1\le j\le n}\left|\frac{\lambda_i}{2-\lambda_j}\right|=|\lambda_2|,
   \]
   Moreover, under the additional assumption $\bv^\top X_0=\bv^\top$ we arrive at
   \[
       \sigma= \max_{2\le i\le n}\max_{2\le j\le n}\left|\frac{\lambda_i}{2-\lambda_j}\right|=\max_{2\le j\le n}\left|\frac{\lambda_2}{2-\lambda_j}\right|.
   \]
\end{proof}

From Theorem \ref{th:convu} we obtain the  entry-wise stable implementation of the U-based iteration \eqref{eq:ub} given in Algorithm \ref{alg:U}.

\begin{algorithm}
    \textbf{Input:} A matrix $N\ge 0$ with $\rho(N)=1$ and simple eigenvalue equal to 1; a vector $\bsu>0$ such that $N\bsu=\bsu$; a matrix $X_0\ge 0$ such that $X_0\bsu=\bsu$; a tolerance $\epsilon>0$\;
    \textbf{Output:} An approximation $Y$ to $(I-N)^{1/2}$, together with the triplet representation $(-\offdiag(Y),\bsu,0)$ of $Y$ \;   
    \For{$k=0,1,2,\ldots$}
        {
        $X_{k+1}=(2I-X_k)^{-1}N$, where Algorithm \ref{alg:gth} is applied to $2I-X_k$, with $\bv=\bsu$ and $B=N$, that is, given the triplet representation $(-\offdiag(X_k),\bsu,\bsu)$ of $2I-X_k$, the triplet representation $(-\offdiag(X_{k+1}),\bsu,\bsu)$ of $2I-X_{k+1}$ is computed\;
        If {$|X_{k+1}-X_k|\le\epsilon|X_{k+1}|$}{\quad exit\;}
        }
    $-\offdiag(Y)=\offdiag(X_{k+1})$, $\hbox{diag}(Y)=-\hbox{diag}(\bsu)^{-1}\offdiag (Y)\bsu$.
    \caption{U-based fixed-point iteration}\label{alg:U}
\end{algorithm}

\section{On the component-wise numerical stability}
We have observed that in a single step of Algorithms \ref{alg:bin} and \ref{alg:U}, numerical cancellation cannot occur; hence, the component-wise relative error is bounded by $\theta u$, where $\theta$ is a constant depending on the matrix size $n$, and $u$ is the machine precision. Indeed, these errors accumulate at each step of the iteration and may grow with the number of iterations. However, it must be said that, unlike CR, functional iterations are self-correcting algorithms; therefore, the errors accumulated at step $k$ are reduced at the subsequent steps. 

If we denote by $e_k$ an upper bound to the modulus of the maximum component-wise relative error at step $k$, and with $\theta u$ an upper bound to the component-wise roundoff error, we get the inequality
\[
   e_k\le \rho  e_{k-1}+\theta u, 
\]
where $\rho<1$ is the error reduction factor at each step.
This yields 
\[
   e_k\le \rho^k e_0+\theta u\sum_{i=0}^{k-2}\rho^i\le \rho^k e_0+\frac {\theta }{1-\rho}u.
\]
That is, the error bound increases as the convergence rate $\rho$ approaches 1.

\section{On the existence of $\bsu>0$ such that $A\bsu =0$ }\label{11}
In the theoretical results of Sections~\ref{sec:natural} and \ref{sec:fixpU}, we make the hypothesis that there exists $\bsu>0$ such that $A\bsu = 0$. This hypothesis is necessary in Theorem~\ref{th:convsingular}; the matrix 
\[
    A = \begin{bmatrix} 1 & 0 \\ 0 & 0\end{bmatrix}=I-N,\quad N=\begin{bmatrix}
        0&0\\0 & 1
    \end{bmatrix},
\]
such that $N\bsu=\bsu$ for $\bsu = \begin{bmatrix} 0 & 1\end{bmatrix}^\top \ge 0$,  and 1 is a simple eigenvalue of $N$, gives an example. In fact, with $X_0 =\left[\begin{smallmatrix} 3 & 0 \\ 2 & 1\end{smallmatrix}\right]$ such that $X_0\bsu = \bsu$, the limit of (2) does not exist. 
To prove this, consider $\bv^\top=\begin{bmatrix}
    1 & 0
\end{bmatrix}$, then we have $\bv^\top X_0=\theta_0\bv^\top$, $\theta_0=3$, and $\bv^\top N=0$. 
This implies that $\bv^\top X_{k+1}=\frac12 \bv^\top X_k^2$, whence $\bv^\top X_{k+1}=\theta_{k+1}\bv^\top$, and $\theta_{k+1}=\frac12\theta_k^2$. 
Since $\theta_0=3$, it follows that $\theta_k=2\left(\frac32\right)^{2^k}$, so that $\lim_{k\to\infty}\theta_k=\infty$.
On the other hand, for the matrix $A$ there exists no positive vector such that $A\bsu = 0$.

The existence of $\bsu>0$ such that $A\bsu=0$ is guaranteed for graph Laplacians, but for a general M-matrix, this property is not generally true. Here, we characterize M-matrices for which there exists $\bsu >0$ such that $A\bsu = 0$, and M-matrices for which there exist $\bsu >0$ and $\bv >0$ such that $A\bsu = 0$ and $\bv^\top A=0$, complementing the results on the existence of $\bsu>0$ such that $A\bsu\ge 0$ in \cite{uno}.

We use the Frobenius normal form of $A$, that is, the block upper triangular matrix
\begin{equation}\label{eq:fnf}
    \Pi^\top  A\Pi = \begin{bmatrix}
    A_{11} & A_{12} & \cdots & A_{\nu\nu}\\
    & A_{22} & \ddots & \vdots\\
    & & \ddots & \vdots\\
    & & & A_{\nu\nu}
    \end{bmatrix},
\end{equation}
where $\Pi$ is a permutation matrix and $A_{ii}$ is a square nonsingular or singular irreducible matrix, for $i=1,\ldots,\nu$ ($1\times 1$ zero blocks are singular irreducible in our notation).

For the sake of simplicity, we assume that $A$ is in its Frobenius normal form, and that the vector $\bsu$ such that $A\bsu=0$ is partitioned according to the block partitioning of $A$ as $\bsu^\top=[\bsu_1^\top,\bsu_2^\top,\ldots,\bsu_\nu^\top]$.

\begin{theorem}\label{1}
   Let $A\in\mathbb R^{n\times n}$ be an M-matrix in Frobenius normal form \eqref{eq:fnf}. There exists $\bsu\in\mathbb R^n$,  $\bsu>0$ such that $A\bsu = 0$ if and only if  there exists at least a singular diagonal block and
   \begin{itemize}
      \item[(i)] if $A_{ii}$ is singular, then $A_{ij}=0$ for $j\ne i$;
      \item[(ii)] if $A_{ii}$ is nonsingular, there exists $j\ne i$ such that $A_{ij}\ne 0$.
   \end{itemize}
\end{theorem}
\begin{proof}
   If $A\bsu = 0$, then the matrix is singular and there exists a singular diagonal block $A_{ii}$. Since $A_{ii}$ is an irreducible singular M-matrix, there exists $\bw>0$ such that $A_{ii}\bw=0$. Therefore,  $(i)$ holds by \cite[Theorem 4]{uno}. Finally, by partitioning $\bsu$ as $\bsu^\top=[\bsu_1^\top,\bsu_2^\top,\ldots,\bsu_\nu^\top]$, if $A_{ii}$ is nonsingular and $A_{ij}=0$ for $j\ne i$, then  we obtain $A_{ii}\bsu_i = 0$, which contradicts the nonsingularity of $A_{ii}$, thus, $(ii)$ holds true.

   Conversely, assume that there exists a singular diagonal block in the Frobenius normal form of $A$ and assume that the properties $(i)$ and $(ii)$ hold. We provide a vector $\bsu>0$ such that $A\bsu=0$ in the following way. 
   Partition ${\bsu}$  into blocks ${\bsu}_i$ according to the partitioning of $A$. Observe that, if $A_{ii}$ is singular and irreducible, then there exists ${\bsu}_i>0$ such that $A_{ii}{\bsu}_i = 0$. Moreover, since $A_{ij}=0$ for $j\ne i$, then for property $(i)$, the $i$th block component of $A{\bsu}$ is zero independently of the values of $ {\bsu}_j$ for $j\ne i$. Now consider the indices $i_1\le \ldots\le i_\ell$ of the nonsingular diagonal blocks of $A$, the condition $(ii)$ implies $i_\ell<\nu$. Since ${\bsu}_j$ is well defined for $j>i_\ell$, we may define ${\bsu}_{i_\ell} := - A_{i_\ell i_\ell}^{-1}(\sum_{j>i_\ell} A_{i_\ell j}{\bsu}_j)$ so that the $i_\ell$-th block component of $ A{\bsu}$ is zero.  We can easily prove that ${\bsu}_{i_\ell}>0$. In fact, for property $(ii)$, at least one of the matrices in the sum is not zero; therefore, the vector $\bz$ obtained from the sum is nonpositive and different from zero. Since 
    $A_{i_\ell i_\ell}^{-1}>0$ because $A_{i_\ell i_\ell}$ is irreducible and invertible, then $-A_{i_\ell i_\ell}^{-1}\bz>0$.
   We may repeat the same argument for $i_{\ell-1}$, \dots, $i_1$ and define the positive vectors ${\bsu}_{i_{\ell-1}},\ldots ,{\bsu}_{i_{1}}$. This way, $ A{\bsu}$ has all its components zero and ${\bsu}>0$.
\end{proof}

From Theorem \ref{1} we get a constraint for the structure of a singular M-matrix with positive left and right eigenvectors.

\begin{corollary}\label{2}
   Let $A\in\mathbb R^{n\times n}$ be an M-matrix. Then there exist $\bsu,\bv\in\mathbb R^n$, such that $\bsu>0$, $\bv>0$ and $A\bsu = 0$, $\bv^\top  A = 0$ if and only if the Frobenius normal form of $A$ is block diagonal with irreducible singular diagonal blocks.
\end{corollary}

Note that in the hypotheses of Corollary~\ref{2},
by using the property that $(S^{-1}AS)^{1/2} = S^{-1}A^{1/2}S$ and that the square root of a block diagonal matrix is block diagonal with the same block partition, we obtain
\[
    A^{1/2} = \Pi\begin{bmatrix} A_{11}^{1/2}\\ & \ddots \\  & & A_{\nu\nu}^{1/2}\end{bmatrix}\Pi^\top.
\]
Therefore,  the component-wise accurate computation of the square root of $A$ can be reduced to the component-wise accurate computation of the square root of its diagonal blocks.

\subsection{The case of a singular matrix with triplet}
As pointed out in the previous section, the condition $A\bsu =0$ for some $\bsu>0$ is needed to apply the algorithms proposed in Sections~\ref{sec:natural} and \ref{sec:fixpU}. Here we show that we can still devise a component-wise stable algorithm under the milder assumption that  there exist $\bsu>0$ and $\bsu'\ge 0$ such that $A\bsu =\bsu'$, i.e., $A$ has a triplet representation. 
In the sequel of the section, it will be useful to refer to a triplet representation as a \emph{right triplet representation}.
Analogously a left triplet can be defined as the triple $(-\offdiag(A^\top),\bv,\bv')$, with $\bv>0$, $\bv'\ge0$ such that $\bv^\top A=(\bv')^\top $ and it is a right triplet of the transpose matrix $A^\top$. 

For an M-matrix, condition $(i)$ of Theorem \ref{1} is equivalent to the existence of a right triplet \cite{uno}, and it is always fulfilled for M-matrices arising in Markov chains and network models.

The condition for the existence of a left triplet is that in the Frobenius normal form 
{\em \begin{itemize}
\item[(i)'] if $A_{ii}$ is nonsingular, then there exists $j\ne i$ such that $A_{ji}\ne 0$.
\end{itemize}
}

Left and right triplet representation can be used to design component-wise accurate algorithms for M-matrix problems (see \cite{mehdi}), including the matrix square root as shown in \cite{uno}.

If a matrix is irreducible, either singular or nonsingular, then necessarily a right and a left triplet exist \cite{uno}, and from a right (left) triplet a left (right) one can be obtained accurately. Indeed, if $A$ is nonsingular irreducible then $(-\offdiag(A^\top), (A^{-1})^\top\bv',\bv')$ is a left triplet if $\bv'\ge 0$, because $A^{-1}>0$ can be computed accurately; if $A$ is singular irreducible then from the LU factorization $A=LU$ we can easily get $\bv>0$ such that $\bv^\top A=0$.

Now we describe how to use the proposed algorithms to compute the component-wise accurate square root of a singular M-matrix $A$ with a (right) triplet.

We use the Frobenius normal form of $A$ and the property that a primary matrix functions $f$ of a block triangular matrix $T=[T_{ij}]$ is block triangular with the same block structure and $(f(T))_{ii}=f(T_{ii})$ \cite[Theorem 1.13(f)]{higham:book}. 

We assume without loss of generality that $A$ is in the Frobenius normal form \eqref{eq:fnf}.
The square root of the block upper triangular matrix $ A$, denoted as $X = [X_{ij}]_{i,j=1,\ldots,\nu}$, is such that $X_{ij} = 0$ for $i>j$, $X_{ii} = A_{ii}^{1/2}$; moreover,  using the equation $X^2 =  A$,  we get 
\begin{equation}\label{3}
    X_{ii}X_{ij}+X_{ij}X_{jj} = A_{ij} - \sum_{\ell = i+1}^{j-1} X_{i\ell}X_{\ell j},\quad \hbox{for } i<j.
\end{equation}
Applying the vec operator that stacks the columns of an $m\times n$ matrix in a vector of length $mn$ and the formula vec$(AXB)=(B^\top \otimes A)$vec$(X)$, we can write \eqref{3} as a linear system with matrix coefficient $M=I\otimes A_{ii}^{1/2}+(A_{jj}^{1/2})^\top \otimes I$ and nonpositive right-hand side including blocks of $X$ on the left and below $X_{ij}$ in the matrix $X$. Thus, computing the diagonal of $X$ and solving \eqref{3} one column at a time, from the diagonal to the top, yields an algorithm for computing $A^{1/2}$.

The matrix $M$ is a Z-matrix with eigenvalues  $\lambda^{(i)}_h+\lambda^{(j)}_k$ where $\lambda^{(i)}_h$ and $\lambda^{(j)}_k$ are eigenvalues of $A_{ii}^{1/2}$ and $A_{jj}^{1/2}$, respectively. Since $A_{ii}^{1/2}$ and $A_{jj}^{1/2}$ are M-matrices, their nonzero eigenvalues have positive real part, therefore, in view of \cite[Theorem 4.6 E11]{bp:book}, $M$ is an M-matrix and the linear system in \eqref{3} is singular if and only if $A_{ii}$ and $A_{jj}$ are singular. By $(i)$ of Theorem \ref{1}, if $A_{ii}$ is singular, then $A_{ij}=0$ for $j\ne i$ and since $X$ is a polynomial of $A$, we get $X_{ij}=0$ for $j\ne i$. Equation \eqref{3} has a zero right-hand side and the solution $X_{ij} = 0$ (with infinitely many non-interesting solutions). The only equations that should be solved are the ones in which $X_{ii}$ is nonsingular, and in this case, the solution is unique.

We summarize the procedure for the component-wise accurate computation of the blocks $X_{ij}$, for $j\ge i$, of $A^{1/2}$, where $A$ is an M-matrix in Frobenius normal form, with triplet representation 
$(-\offdiag(A),\bsu,\bsu')$, with 
$\bsu^\top = \begin{bmatrix}\bsu_1^\top & \cdots & \bsu_\nu^\top\end{bmatrix}  $ partitioned according to the block structure of $A$:

\begin{itemize}

\item Compute the triplets of the diagonal blocks $A_{ii}$, for $i=1,\ldots,\nu$:
\begin{itemize}
    \item if $A_{ii}$ is singular then the triplet is $(-\offdiag(A_{ii}),\bsu_i,0)$,
    \item if $A_{ii}$ is nonsingular then the triplet is $(-\offdiag(A_{ii}),\bsu_i,\bsu''_i)$, with $\bsu''_i= \bsu'_i-\sum_{j> i} A_{ij}\bsu_j$.
\end{itemize} 
\item Compute the triplets $(-\offdiag(A_{ii}^{1/2}),\bsu_{i,1/2},\bsu'_{i,1/2})$ 
for $i=1,\ldots,\nu$, by using the algorithm of \cite{uno} if $A_{ii}$ is nonsingular, or an algorithm of Section~\ref{sec:natural} or~\ref{sec:fixpU} if $A_{ii}$ is singular.

\item Set $X_{ij} = 0$ if $A_{ii}$ is singular and $j>i$.
\item Compute $X_{ij}$ if $A_{ii}$ is nonsingular and $j>i$ as follows:
\begin{itemize}
    \item  compute a left triplet representation of $A_{jj}^{1/2}$, namely
    \[
    (-\offdiag(A_{jj}^{1/2})^\top ,\bv_{j,1/2},\bv'_{j,1/2});
    \]

    \item solve the linear system obtained by \eqref{3}, i.e.,
\[
   M  \opvec(X_{ij}) = \opvec\bigl(A_{ij}-\sum_{\ell=i+1}^{j-1} X_{i\ell}X_{\ell j}\bigr),
\]
where $M=(I \otimes A_{ii}^{1/2} + (A_{jj}^{1/2})^\top  \otimes I)$; this system is solved 
accurately by means of Algorithm~\ref{alg:gth} by observing that
\[
(-\offdiag(M),\bv_{j,1/2}\otimes \bsu_{i,1/2},\bv_{j,1/2}\otimes \bsu'_{i,1/2}+\bv'_{j,1/2}\otimes \bsu_{i,1/2})
\]
is a triplet representation of $M$.
\end{itemize}

\end{itemize}

This way, the problem of accurate computation of the square root of a singular M-matrix with a triplet can be reduced to the computation of square roots of nonsingular or singular irreducible M-matrices, the left triplet of an irreducible M-matrix, and the solution of linear systems with M-matrix coefficients.

\section{Numerical tests}\label{sec:exp}
In this section, we present results from numerical experiments comparing the proposed fixed-point iterations with CR for graph Laplacians.
We recall that a direct graph $\mathcal G$ having $n$ nodes can be assigned in terms of its $n\times n$ adjacency matrix $W=(w_{ij})$ where $w_{ij}$ is the weight of the direct edge which connects the node $i$ to the node $j$ in $\mathcal G$. Indeed, if $w_{ij}=0$ then there is no edge connecting node $i$ to node $j$. The Laplacian matrix of $\mathcal G$ is defined as $L=D-W$ where $D$ is the diagonal matrix with diagonal entries $(d_1,\ldots,d_n)$, where $d_i=\sum_{j=1}^n w_{ij}$ is the degree of the $i$th node. We have written the Laplacian matrix $L$ in the form $L=\alpha(I-N)$ where
$\alpha=\max_i d_i$. Finally, we have computed $(I-N)^{1/2}=I-G$. 

As a set of tests, we have considered adjacency matrices of graphs for which each column has at least one zero entry. In this case, the acceleration technique valid for CR, analyzed in \cite{uno}, cannot be applied, and standard CR, in the component-wise stable version given in \cite{uno}, has linear convergence with rate $1/2$.
 In our tests, the smallest modulus entry in the matrix $G$ is in the range $[10^{-14},10^{-2}]$; the smaller this value, the more inaccurate is the result computed by standard algorithms for the square root.
The experiments have been performed in Matlab  
Version: 9.11.0 (R2021b) on a laptop with an Intel I3 CPU and Ubuntu  Release 20.04.6 LTS 64-bit.
The adjacency matrix $W$ has been generated in the following way, with the additional condition that $w_{ii}=0$, i.e., the graph $\mathcal G$ has no self-loops.

\begin{enumerate}
    \item Banded matrix $W$ obtained with the commands
    
    {\tt T = rand(n); p = floor(n/6);
        W = tril(triu(T,-p),n-p-2);}
        
        This matrix has a zero entry in each column and in each row; moreover, the minimum modulus entry  of the square root of the corresponding matrix $G$ has the order of $10^{-8}$.
    
    \item The matrix $W$ is defined as in the previous test with  {\tt p = floor(n/10).} In this case, 
    the minimum modulus entry of $G$ has the order of $10^{-13}$.
    
    \item Random bipartite graph where, for $n=2m$, the adjacency matrix $W$ has the form {\tt W = [Z, B1;B2, Z];} with {\tt Z =  zeros(m);} {\tt B1 = rand(m);} {\tt B2 = rand(m);}
    In this case, the minimum modulus entry of $G$ has the order $10^{-3}$.
   
   \item Bipartite graph obtained as before, but where the off-diagonal blocks $\tt B1$ and $\tt B2$ are in generalized upper Hessenberg form and are generated by the command 
   {\tt triu(T,-floor(m/p));} where {\tt T = rand(m);} and {\tt p = 10}.
   In this case, the minimum modulus entry of the square root of the corresponding matrix has the order $10^{-13}$.
   
\end{enumerate}

Randomization has been initialized with the command {\tt rng(0)}. In our tests, we have compared the natural and U-based fixed-point iterations described in Algorithms \ref{alg:bin} and \ref{alg:U}, with CR in the triplet-based version of \cite{uno}, here denoted with CR,  and with standard CR with the shift acceleration \cite{blm:book}, denoted with S-CR. While the triplet-based version of CR is component-wise numerically stable, the standard shifted version of CR is norm-wise stable and does not guarantee a uniform component-wise bound to the error close to the machine precision.  In the experiments,  we compare and report on the number of iterations, the CPU time, and the maximum component-wise relative errors of each algorithm. The component-wise relative error bounds are also compared to the ones generated by the Matlab command {\tt sqrtm}.

In Table \ref{tab:rho}, we display the values of the convergence rate, as given by Theorems \ref{th:convsingular} and \ref{th:convu}, of Algorithms \ref{alg:bin}, \ref{alg:U}, and CR, for the matrices used in our tests 1--4 with $n=100$. 
Concerning CR, we observe that for a singular M-matrix, the convergence rate is $\frac{1}{2}$. This follows from the fact that equation \eqref{eq:qme}, suitably scaled, can be viewed as a null recurrent QBD, and for this class of problems, the convergence rate of CR is $\frac{1}{2}$, as shown in   \cite{guo}.

The table also reports the maximum and the minimum values of the moduli of the entries of $G$. We recall that the smaller the minimum value, the larger the component-wise relative errors are expected in the computed matrix $G$.

\begin{table}[]
	\centering	
	\begin{filecontents*}{rate_100.dat}
Test, Alg. 2, Alg. 3, CR, min, max 
1, 0.62, 0.45, 0.50, 2.8e-08, 1.0 
2, 0.68, 0.52, 0.50, 1.8e-13, 1.0 
3, 0.29, 0.13, 0.50, 3.0e-03, 1.0 
4, 0.69, 0.53, 0.50, 5.1e-13, 1.0 
\end{filecontents*}
\pgfplotstabletypeset[
	col sep = comma,
	header = has colnames,
	columns = {Test, Alg. 2, Alg. 3, CR, min, max},
	every head row/.style={after row=\hline},
	every first column/.style={column type/.add={}{|}},
	columns/Test/.style = {string type, column name = Test},
	columns/Alg. 2/.style   = {column name = Alg. 1, fixed, precision = 2},
	columns/Alg. 3/.style  = {column name = Alg. 2, fixed, precision = 2},
	columns/CR/.style  = {column name = CR, fixed, precision = 2},
	columns/min/.style  = {string type, column name = min},
	columns/max/.style  = {column name = max, fixed, precision = 2}        
	]{rate_100.dat}
\caption{Convergence rates of three algorithms for Tests 1, 2, 3, 4, as provided by Theorems~\ref{th:convsingular} and \ref{th:convu} for $X_0=I$, together with minimum and maximum modulus of the entries of $A^{1/2}$. The numerical values are the medians out of 50 samples generated randomly with $n=100$.}
\label{tab:rho}
\end{table}	

Since for Test 3 the convergence rate of Algorithms 1 and 2 is much smaller than 0.5, for these algorithms, we expect a better performance than CR. For Tests 1, 2, and 4, where the rate $\rho$ is slightly greater than 0.5, the larger number of iterations needed by Algorithms 1 and 2 should be compensated by the smaller computational cost per step.
We will appreciate this fact in the experiments described in the next subsection.

\subsection{Comparing natural and U-based iteration to CR}

Table \ref{tab:pwer} reports the maximum component-wise relative error obtained by the different methods for the different tests, together with the number of iterations needed by the algorithms. 
It is interesting to observe that, as expected, the result given by the Matlab command {\tt sqrtm} produces a large component-wise relative error; this is more evident for Tests 2 and 4 where the minimum absolute value of the entries of $G$ takes the smallest values (compare with Table \ref{tab:rho}).

\begin{table}[]
	\centering		
\begin{filecontents*}{mxerr-iter_100.dat}
Test, Alg. 2, Alg. 3, CR, S-CR, sqrtm
1, 4.30e-15 (86), 4.66e-15 (52), 3.06e-15 (52), 3.49e-11 (8), 5.28e-08
2, 4.29e-15 (109), 3.73e-15 (65), 2.33e-15 (54), 2.42e-06 (8), 4.42e-03
3, 1.66e-15 (31), 1.73e-15 (21), 1.70e-15 (53), 1.43e-15 (7), 4.92e-08
4, 9.82e-15 (113), 3.21e-15 (65), 2.66e-15 (55), 9.81e-07 (9), 7.34e-04
\end{filecontents*}
\pgfplotstabletypeset[
  col sep = comma,
  header = has colnames,
  columns = {Test, Alg. 2, Alg. 3, CR, S-CR, sqrtm},
  every head row/.style={after row=\hline},
  every first column/.style={column type/.add={}{|}},
  columns/Test/.style = {string type, column name = Test},
  columns/Alg. 2/.style   = {string type, column name = Alg. 1},
  columns/Alg. 3/.style  = {string type, column name = Alg. 2, string type},
  columns/CR/.style  = {string type, column name = CR, string type},
  columns/S-CR/.style  = {string type, column name = S-CR},
  columns/sqrtm/.style  = {string type, column name = sqrtm}       
]{mxerr-iter_100.dat}
\caption{Maximum component-wise relative errors generated by the different algorithms in our tests performed with $n=100$. Between parentheses, the number of iterations needed. In the last column, the relative error generated by the Matlab command {\tt sqrtm}.}
\label{tab:pwer}
\end{table}

A second remark concerns the algorithm S-CR; in fact, it is the one that takes the smallest number of iteration steps due to its quadratic convergence. However, since this algorithm is norm-wise stable but not component-wise stable, the maximum component-wise relative errors generated by S-CR take pretty large values even though smaller than the ones generated by the Matlab command {\tt sqrtm}. On the other hand, Algorithms \ref{alg:bin}, \ref{alg:U}, and CR produce relative errors close to the machine precision that differ little from each other. The number of iterations required by these algorithms, reported between parentheses, agrees with the convergence rates reported in Table \ref{tab:rho}.

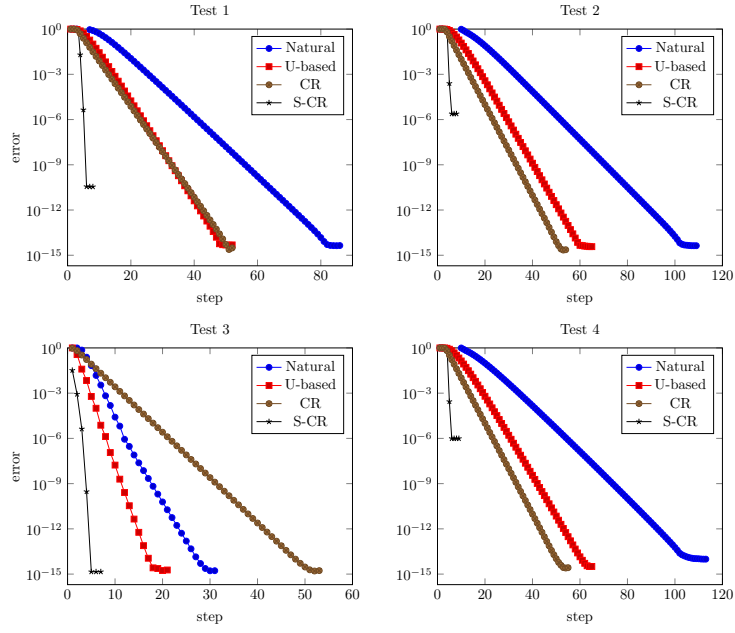
\begin{figure}
\begin{center}	
\begin{tabular}{cc}		
\begin{tikzpicture} [trim axis left, trim axis right, scale=0.55]
    \begin{axis}[
    xlabel={step},
    ylabel={error},
    legend pos=north east,
    ymode = log,
    log base y=10,
    ymin = 2e-16, ymax = 1,
    xmin = 0, xmax = 90,
    title = Test 1,
    ]
        \addplot table[x=step, y=nat] {err_plot_test1.dat};
        \addplot table[x=step, y=ubs] {err_plot_test1.dat};
        \addplot table[x=step, y=cr] {err_plot_test1.dat};
        \addplot table[x=step, y=crs] {err_plot_test1.dat};
        \legend{Natural, U-based, CR, S-CR}
    \end{axis}
\end{tikzpicture}
&\qquad
\begin{tikzpicture} [trim axis left, trim axis right, scale=0.55 ]
\begin{axis}[
xlabel={step},
legend pos=north east,
ymode = log,
log base y=10,
ymin = 2e-16, ymax = 1,
xmin = 0, xmax = 120,
title = Test 2,
]
\addplot table[x=step, y=nat] {err_plot_test2.dat};
\addplot table[x=step, y=ubs] {err_plot_test2.dat};
\addplot table[x=step, y=cr] {err_plot_test2.dat};
\addplot table[x=step, y=crs] {err_plot_test2.dat};
\legend{Natural, U-based, CR, S-CR}
\end{axis}
\end{tikzpicture}\\
\begin{tikzpicture} [trim axis left, trim axis right, scale=0.55]
\begin{axis}[
xlabel={step},
ylabel={error},
legend pos=north east,
ymode = log,
log base y=10,
ymin = 2e-16, ymax = 1,
xmin = 0, xmax = 60,
title = Test 3,
]
\addplot table[x=step, y=nat] {err_plot_test3.dat};
\addplot table[x=step, y=ubs] {err_plot_test3.dat};
\addplot table[x=step, y=cr] {err_plot_test3.dat};
\addplot table[x=step, y=crs] {err_plot_test3.dat};
\legend{Natural, U-based, CR, S-CR}
\end{axis}
\end{tikzpicture}&\qquad 
\begin{tikzpicture} [trim axis left, trim axis right, scale=0.55]
\begin{axis}[
xlabel={step},
legend pos=north east,
ymode = log,
log base y=10,
ymin = 2e-16, ymax = 1,
xmin = 0, xmax = 120,
title = Test 4,
]
\addplot table[x=step, y=nat] {err_plot_test4.dat};
\addplot table[x=step, y=ubs] {err_plot_test4.dat};
\addplot table[x=step, y=cr] {err_plot_test4.dat};
\addplot table[x=step, y=crs] {err_plot_test4.dat};
\legend{Natural, U-based, CR, S-CR}
\end{axis}
\end{tikzpicture}
\end{tabular}\caption{Maximum component-wise relative error at each step of the different algorithms. From top left to bottom right, the case of Tests 1,2,3,4. Here, CR is cyclic reduction as given in \cite{uno} in the version relying on triplets, while S-CR is standard cyclic reduction with the shift acceleration \cite{blm:book}.}
    \label{fig:2-6-100}
\end{center}
\end{figure}

These features are confirmed in Figure \ref{fig:2-6-100}, where the dynamics of convergence and the approximation error obtained at each step by the different algorithms are reported on a semi-log scale.
Here, the algorithms are denoted as {\tt natural}, {\tt U-based}, {\tt CR}, {\tt S-CR}.
We may observe that while shifted CR has the steepest descent due to quadratic convergence, the approximation error cannot go below a large threshold due to the lack of component-wise stability. For all the other algorithms, the graphs show a linear descent and reach values close to machine precision. Except for Test 3, where U-based and natural algorithms have much faster convergence than CR, in the other tests, Algorithm \ref{alg:bin} is the slowest one, and Algorithm \ref{alg:U} has a convergence speed close to that of CR.

Concerning the CPU time, Figure \ref{fig:time} reports the values taken by Algorithm \ref{alg:bin}, 
Algorithm \ref{alg:U}, 
and CR, 
for the four different tests. 
The numerical values are the median computed from 50 runs of each algorithm using the same test matrix.

We may observe that, despite the number of iterations needed by Algorithms~\ref{alg:bin} and \ref{alg:U} being larger than the one needed by CR for Tests 1,2, and 4, the CPU time of Algorithms~\ref{alg:bin} and \ref{alg:U} is substantially inferior to the time needed by CR in the triplet-based version of~\cite{uno}. This difference is even more evident for Test~3, where the convergence of Algorithms~\ref{alg:bin} and \ref{alg:U} is faster.

The better efficiency of Algorithms \ref{alg:bin} and \ref{alg:U} is mainly because the cost per iteration is lower than the cost of the other algorithms. 
In fact, each step of Algorithm \ref{alg:bin} requires only one matrix multiplication and one matrix addition. Both computations are built-in operations in Matlab and rely on the BLAS package. Whereas the version of CR given in \cite{uno}, based on triplets, requires using the GTH trick for inverting an M-matrix that is not much more expensive than matrix multiplication, but is implemented as a separate Matlab function. Its execution is slower than matrix operations that rely on BLAS.

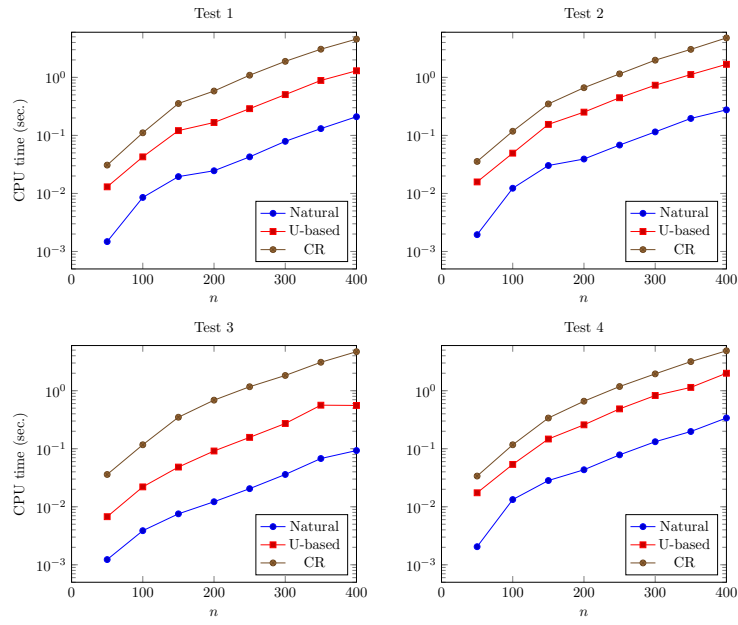
\begin{figure}
\begin{center}	
\begin{tabular}{cc}		
\begin{tikzpicture} [trim axis left, trim axis right, scale=0.55]
    \begin{axis}[
    xlabel={$n$},
    ylabel={CPU time (sec.)},
    legend pos=south east,
    ymode = log,
    log base y=10,
    ymin = 5e-4, ymax = 6,
    xmin = 0, xmax = 400,
    title = Test 1,
    ]
        \addplot table[x=n, y=nat] {cpu-test1.dat};
        \addplot table[x=n, y=ubs] {cpu-test1.dat};
        \addplot table[x=n, y=cr] {cpu-test1.dat};
        \legend{Natural, U-based, CR}
        
    \end{axis}
\end{tikzpicture}
&\qquad
\begin{tikzpicture} [trim axis left, trim axis right, scale=0.55]
    \begin{axis}[
    xlabel={$n$},
    legend pos=south east,
    ymode = log,
    log base y=10,
    ymin = 5e-4, ymax = 6,
    xmin = 0, xmax = 400,title = Test 2,
    ]
        \addplot table[x=n, y=nat] {cpu-test2.dat};
        \addplot table[x=n, y=ubs] {cpu-test2.dat};
        \addplot table[x=n, y=cr] {cpu-test2.dat};
        \legend{Natural, U-based, CR}
    \end{axis}
\end{tikzpicture}\\
\begin{tikzpicture} [trim axis left, trim axis right, scale=0.55]
    \begin{axis}[
    xlabel={$n$},
    ylabel={CPU time (sec.)},
    legend pos=south east,
    ymode = log,
    log base y=10,
    ymin = 5e-4, ymax = 6,
    xmin = 0, xmax = 400,title = Test 3,
    ]
        \addplot table[x=n, y=nat] {cpu-test3.dat};
        \addplot table[x=n, y=ubs] {cpu-test3.dat};
        \addplot table[x=n, y=cr] {cpu-test3.dat};
        \legend{Natural, U-based, CR}
    \end{axis}
\end{tikzpicture}&\qquad 
\begin{tikzpicture} [trim axis left, trim axis right, scale=0.55]
    \begin{axis}[
    xlabel={$n$},
    legend pos=south east,
    ymode = log,
    log base y=10,
    ymin = 5e-4, ymax = 6,
    xmin = 0, xmax = 400,title = Test 4,
    ]
        \addplot table[x=n, y=nat] {cpu-test4.dat};
        \addplot table[x=n, y=ubs] {cpu-test4.dat};
        \addplot table[x=n, y=cr] {cpu-test4.dat};
        \legend{Natural, U-based, CR}
    \end{axis}
\end{tikzpicture}
\end{tabular}\caption{CPU time needed by the different algorithms for different values of $n$. From top left to bottom right, the results of Tests 1, 2, 3, 4.}
\label{fig:time}
\end{center}
\end{figure}

It must be said that Algorithm \ref{alg:U} requires one matrix inversion per step, which is performed using the same function that implements the GTH trick. However, in this case, due to the self-correcting feature of Algorithm \ref{alg:U}, we alternated between steps that used inversion via the faster ``backslash'' command and steps that used the GTH trick. This strategy maintains the same convergence and preserves the component-wise accuracy of the computation, but with a lower overall CPU time.

\section{Conclusions}\label{sec:conclusions}
We have provided and analyzed two fixed-point iterations to compute the square root of a singular M-matrix $A$ in a component-wise, numerically stable way.
The convergence analysis of these algorithms has been performed, and it has been shown that with a suitable choice of the initial approximation, the linear convergence can be preserved and the convergence factor can be explicitly expressed in terms of the eigenvalues of the matrix $A$.

Numerical experiments demonstrate the effectiveness of these fixed-point iterations for certain classes of problems.

\end{document}